\documentclass{proc-l}
\usepackage{hyperref}
\usepackage{xcolor}
\usepackage{amsmath}
\usepackage[T1]{fontenc}
\usepackage{listings}
\usepackage{amssymb}
\usepackage{enumitem}

\RequirePackage{tikz-cd}

\newcommand{\R}{\mathbf{R}}

\newcommand{\Z}{\mathbf{Z}}

\newcommand{\N}{\mathbf{N}}
\newcommand{\Q}{\mathbf{Q}}

\newcommand{\F}{\mathbf{F}}

\newcommand{\G}{\mathbf{G}}

\newcommand{\base}{\mathbf{b}}

\newcommand{\cone}{\text{C}_{\text{eff}}(X)}

\DeclareMathOperator{\Pic}{Pic}

\DeclareMathOperator{\rank}{rank}

\DeclareMathOperator{\Vol}{Vol}

\DeclareMathOperator{\stack}{stack}

\DeclareMathOperator{\num}{num}

\DeclareMathOperator{\age}{age}

\DeclareMathOperator{\coarse}{coarse}

\DeclareMathOperator{\orb}{orb}

\newcommand{\stackS}{\mathcal{S}}
\newcommand{\stackT}{\mathcal{T}}

\newcommand{\stackP}{\mathcal{P}}
\newcommand{\affine}{\mathbf{A}}

\newcommand{\sheafI}{\mathcal{I}}
\newcommand{\sheafO}{\mathcal{O}}

\newcommand{\w}{\omega}

\newcommand{\twistsector}{\pi_0^*(\sheafI_{\mu} X)}

\newtheorem{theorem}{Theorem}[section]
\newtheorem{cor}[theorem]{Corollary}

\newtheorem{lemma}[theorem]{Lemma}
\newtheorem{prop}[theorem]{Proposition}

\theoremstyle{definition}
\newtheorem{definition}[theorem]{Definition}
\newtheorem{hypothesis}[theorem]{Hypothesis}

\theoremstyle{remark}
\newtheorem{remark}[theorem]{Remark}
\newtheorem{question}[theorem]{Question}

\numberwithin{equation}{section}

\begin{document}

\title{The multi-height distribution implies the Batyrev-Manin principle}


\author{Bongiorno Nicolas}
\address{}
\curraddr{}
\email{nicolas.bongiorno@univ-grenoble-alpes.fr}
\thanks{}

\subjclass[2020]{}

\date{}

\dedicatory{}


\begin{abstract}
We explain how to deduce from the multi-height analysis of rational points on a toric stack (respectively on a toric variety) the asymptotic study of the number of rational points of bounded orbifold anticanonical height (respectively bounded anticanonical height), using a general version of the hyperbola method developed by Marta Pieropan and Damaris Schindler.
\end{abstract}


\maketitle

\section{Introduction}

\subsection{Background and previous results}

For a quasi-Fano variety $V$, i.e. one satisfying \cite[Hypothesis 3.27]{Peyre_beyond_height} and with infinitely many rational points, one may study asymptotically the finite set of rational points of bounded height $H$ associated with the anticanonical line bundle $\w_{V}^{-1}$. In \cite{FrankeManinTschinkel1989}, \cite{Batyrev1990} and \cite{peyre_duke}, Batyrev, Franke, Manin, Tschinkel, and Peyre provided strong evidence supporting conjectures relating the asymptotic behaviour of the number of rational points of bounded height on a dense open subscheme of $V$ to geometric invariants of $V$. Batyrev and Tschinkel proved these conjectures using harmonic analysis methods in \cite{batyrev_toric_varieties} for proper, smooth, split toric varieties, while Salberger proved them via a universal torsor lifting method in \cite{salberger_torsor}.

Recently, Ellenberg, Satriano, and Zureick-Brown, in \cite{ellenberg2022heights},
initiated a program aimed at providing a common framework for Malle’s conjecture
and the Batyrev--Manin principle. Subsequently, Darda and Yasuda were able to
extend the formalism of Arakelov heights on varieties to Deligne--Mumford (DM)
stacks and to formulate precise conjectures analogous to the Batyrev--Manin
principle, provided one considers rational points of bounded height $H$, where
$H$ is the height associated with the orbifold anticanonical line bundle (see
\cite[Definition~4.3, Definition~9.1 and Conjecture~5.6]{darda2024batyrevmanin}).
This new formalism notably allowed a reinterpretation of number-theoretic work on Malle and Bhargava's conjectures on the distribution of Galois extensions with fixed group $G$ (see \cite{Malle2002,Malle2004} and \cite{bhargava_mass_formula}) through the lens of Manin's program for $BG$. Subsequently, in \cite{loughran_santens_malle_conjecture}, Santens and Loughran further reinterpreted these number-theoretic problems through the study of rational points of $BG$ using this new machinery. They defined a Tamagawa measure associated with $BG$ and formulated a conjecture on what should be the leading constant in the asymptotic behaviour of rational points of bounded height on $BG$ (see \cite[Definition 8.7 and Conjecture 9.1]{loughran_santens_malle_conjecture}). Darda and Yasuda also established the asymptotic behaviour of the number of rational points of bounded orbifold anticanonical height on toric stacks in \cite[Theorem 1.1]{darda_yasuda_toric_stacks_batyrev}.

Peyre recently proposed in \cite[Question 4.8]{Peyre_beyond_height} a new approach to studying the asymptotic distribution of rational points of bounded height. Instead of considering a single height, it is natural to consider all possible heights. The author's doctoral work consisted in establishing that the expected asymptotic behaviour is obtained for toric varieties (see \cite[Theorem 2.15]{bongiorno2024multiheightanalysisrationalpoints}) and for toric stacks (see \cite[Theorem 1.1]{bongiorno_toric_stacks}).

The purpose of this article is to show that, once the asymptotic behaviour of rational
points with respect to multi-heights is known, one can recover the asymptotic formula
for the number of rational points of bounded anticanonical height (or bounded orbifold
anticanonical height in the stacky setting).

In particular, although it was already understood that the leading constant obtained in
\cite[Theorem~2.15]{bongiorno2024multiheightanalysisrationalpoints} coincides with the
one appearing in \cite[Theorem~11.49]{salberger_torsor}, the present approach makes it
possible to identify the constant obtained in
\cite[Corollary~4.4.5]{darda_yasuda_toric_stacks_batyrev} with the orbifold Tamagawa
number of a toric stack, as defined in
\cite[Definition~6.51,Theorem 6.52]{bongiorno_toric_stacks}.

This work essentially relies on the more general version of the hyperbola method developed by  Pieropan and Schindler in \cite[Section 4]{pieropan_hyperbola_method}.

\subsection{Terminology and the precise statement}

Let $X$ be a smooth and complete split toric variety, respectively toric stack, over $\Q$. We write $T \subset X$ for its dense open (stacky) torus. For the remainder of the article, we refer the reader to \cite[Section 2.3]{bongiorno2024multiheightanalysisrationalpoints} and \cite[Section 5]{bongiorno_toric_stacks} for further details on toric varieties and toric stacks. To simplify the reading of this article and to use the same notation for both varieties and toric stacks, we adopt the following notations to denote the geometric objects associated with $X$ when it is a toric stack:
\begin{itemize}
    \item we denote by $\Pic(X)$ the orbifold Picard group of $X$ (see \cite[Definition 3.7]{bongiorno_toric_stacks});
    \item we denote by $\cone$ the cone of orbifold effective divisors of $X$ (see \cite[Definition 3.17]{bongiorno_toric_stacks});
    \item we denote by $\w_X^{-1}$ the orbifold anticanonical line bundle of $X$ (see \cite[Definition 3.14]{bongiorno_toric_stacks});
    \item we denote by $\stackT$ the extended universal torsor (see \cite[Definition 5.34, Definition 6.1]{bongiorno_toric_stacks});
    \item we write $\Sigma(1)$ for the union of the set of rays of a fan associated to $X$ and the set of twisted sectors
(this corresponds to the notation $\Sigma_{\mathrm{ext}}(1)$ introduced
in \cite[Notations 5.33]{bongiorno_toric_stacks} and used subsequently);
    \item for $\lambda \in \Sigma(1)$, we denote by $[D_\lambda]$ the corresponding class in the orbifold Picard group (see \cite[Notation 3.9,5.26,5.38]{bongiorno_toric_stacks}).
\end{itemize}If $X$ is a toric variety, $\Pic(X)$ denotes its usual Picard group, $\cone$ its cone of effective divisors, $\w_X^{-1}$ its anticanonical line bundle and $\stackT$ its universal torsor. Moreover, $\Sigma(1)$ denotes the set of rays of a fan associated with $X$,
and for $\lambda \in \Sigma(1)$, $[D_\lambda]$ denotes the class of the
boundary divisor in the Picard group.
With this notation, we set $\rho(X) = \rank_{\Z}( \Pic(X) ) $, or simply $\rho$.

We equip $X$ (respectively $X^{\coarse}$) with its natural system of heights as in \cite[Definition 9.2]{salberger_torsor} and \cite[Proposition 2.1.2]{Batyrev1990}. If $X$ is a toric stack, we endow it moreover with its natural stacky data (see \cite[Definition 3.25 and Section 6]{bongiorno_toric_stacks}). With these conventions, we define the multi-height map $$h : X(\Q) \rightarrow \Pic(X)^{\vee}_{\R}$$ as in \cite[Definition 2.7]{bongiorno2024multiheightanalysisrationalpoints} for toric varieties and as in \cite[Definition 3.26]{bongiorno_toric_stacks} for toric stacks. We also define a local height on the universal torsor at a given place $v \in M_{\Q}$,
$$
h_{\stackT,v} : \stackT(\Q_v) \longrightarrow \Pic(X)^{\vee}_{\R}.
$$
as in \cite[Definition 3.6]{bongiorno2024multiheightanalysisrationalpoints} for toric varieties and as in \cite[Definition 6.20]{bongiorno_toric_stacks} for toric stacks

We define a measure $\nu$ on $\Pic(X)^{\vee}_{\R}$ as follows:

\begin{definition}\label{mesure_picard_group}
    $$ \nu(\mathrm{D}) = \int\limits_{\mathrm{D}} e^{\langle \w_{X,}^{-1},y \rangle} dy $$for any compact subset $\mathrm{D}$ of $\Pic(X)^{\vee}_{\R}$, where the Haar measure $dy$ on $\Pic(X)^{\vee}_{\R}$ is normalized so that the covolume of the dual lattice of the Picard group is one.
\end{definition}

For any domain $\mathrm{D} \subset \Pic(X)^{\vee}_{\R}$, for any subset $W \subset X(\Q)$, we write:
$$ W_{h \in \mathrm{D}} = \{ P \in W \mid h(P) \in \mathrm{D} \} .$$

We have proven in \cite[Theorem 4.16]{bongiorno2024multiheightanalysisrationalpoints} and in \cite[Theorem 7.16]{bongiorno_toric_stacks} the following theorem:

\begin{theorem}\label{theorem_multi_height}
Let $\mathrm{D}_1$ be a finite union of compact polyhedra of $\Pic(X)^{\vee}_{\R}$,
and let $u$ be an element of the interior of the dual of the effective cone
$(\cone^{\vee})^{\circ}$.
For a real number $B > 1$, we set:
$$ \mathrm{D}_B = \mathrm{D}_1 + \log(B) u .$$
 
Then the multi-height asymptotic behaviour is of the form:
$$ \sharp ( X(\Q) )_{h \in \mathrm{D}_B} \underset{B \rightarrow +\infty}{\sim} \nu(\mathrm{D}_1 ) \tau(X) B^{\langle \w_{X}^{-1},u \rangle}$$ where $\tau(X)$ is the Tamagawa number of $X$ (see \cite[Definition 3.22]{bongiorno2024multiheightanalysisrationalpoints} and \cite[Definition 6.51]{bongiorno_toric_stacks}).

\end{theorem}

The purpose of this article is to show how this theorem leads to the asymptotic study of rational points of bounded anticanonical height, in line with the classical Manin program. To accomplish this, we will use the more general version of the hyperbola method developed by Pieropan and Schindler. Let us recall their result \cite[Theorem 1.1]{pieropan_hyperbola_method}:

\begin{theorem}\label{theorem_hyperbola_method}
Let $f : \N^{\rho} \rightarrow \R_+$ be an arithmetic function satisfying the following property:
there exist strictly positive real numbers $C_{f,M} \leqslant C_{f,E}$, $\Delta > 0$, and $\w_i > 0$ for $1 \leqslant i \leqslant \rho$ such that, 
for $B_1, \dots, B_{\rho} \in \R_{\geqslant 1}$, we have
$$
\sum\limits_{1 \leqslant y_i \leqslant B_i,\ 1 \leqslant i \leqslant \rho} f(y) 
= C_{f,M} \prod\limits_{i = 1}^{\rho} B_i^{\w_i} 
+ O \!\left( C_{f,E} \prod\limits_{i = 1}^{\rho} B_i^{\w_i} 
\min\limits_{1 \leqslant i \leqslant {\rho}}(B_i)^{-\Delta} \right),
$$
where the implied constant is independent of $f$.  

Let $\mathcal{K}$ be a finite set and $\{\alpha_{i,k}\}_{1 \leqslant i \leqslant s,\, k \in \mathcal{K}}$ a family of positive real numbers.  
We define
$$
\mathcal{D}(B,\alpha) = 
\left\{ y \in \N^{\rho} \Big\vert \prod\limits_{i = 1}^{\rho} y_i^{\alpha_{i,k}} \leqslant B \ \forall k \in \mathcal{K} \right\}.
$$
We wish to evaluate the sum
$$
S_f(B) = \sum\limits_{y \in \mathcal{D}(B,\alpha)} f(y).
$$
Let $\stackP$ denote the polyhedron consisting of all $t \in \R_+^{\rho}$ satisfying
$$
\sum\limits_{i = 1}^{\rho} \alpha_{i,k} \w_i^{-1} t_i \leqslant 1 
\quad \forall k \in \mathcal{K}.
$$
If $\mathcal{P}$ is bounded, not contained in a hyperplane, and if the face $F$ on which the function 
$\sum\limits_{1 \leqslant i \leqslant \rho} t_i$ attains its maximal value $a$ is not contained in a coordinate hyperplane of $\R^{\rho}$, 
then, writing $k = \dim(F)$, we have
$$
S_f(B) = (s - 1 - k)! \, C_{f,M} \, c_{\mathcal{P}} \, \log(B)^k B^{a}
+ O \!\left( C_{f,E} \, \log(\log(B))^s \, \log(B)^{k-1} B^a \right),
$$
where $c_{\mathcal{P}}$ is an explicit constant.
\end{theorem}

Let us consider for a real number $B > 1$, the set $$T(\Q)_{h(P)(\w_X^{-1}) \leqslant \log(B)} = \{ P \in T(\Q) \mid h(P)(\w_X^{-1}) \leqslant \log(B) \} .$$Using their theorem, we recover the known result of \cite[Theorem 11.49]{salberger_torsor} and \cite[Theorem 4.4.4]{darda_yasuda_toric_stacks_batyrev}:

\begin{theorem}\label{theorem_article}
    We have the following asymptotic behaviour: $$\sharp T(\Q)_{h(P)(\w_X^{-1}) \leqslant \log(B)}   \underset{B \rightarrow +\infty}{\sim} \alpha(X) \tau(X) B^{\langle \w_{X}^{-1},u \rangle} \log(B)^{\rho(X) - 1} $$where $\alpha(X)$ is the constant given by \ref{definition_effective_leading_constant}.
\end{theorem}

We can deduce the following corollary about the leading constant in the asymptotic distribution of rational points of toric stacks:

\begin{cor}
    For $X$ a toric stack, the Tamagawa number $\tau(X)$ of a split toric stack $X$ defined by the author in \cite[Definition 6.51]{bongiorno_toric_stacks} corresponds to the constant obtained by Darda and Yasuda in \cite[Theorem 4.4.4]{darda_yasuda_toric_stacks_batyrev}. That is to say, we have the equality:
$$(\rho - 1)! \ \w_H(X(\affine_{\Q})) = \prod_{p}
\left( 1 - \frac{1}{p} \right)^{\rank \Pic_{\orb}(X)}
\cdot
\mu_{X,p}(X(\Q_p)),
$$ where $\mu_{X,p}$ is a Tamagawa measure on $X(\Q_p)$ (see \cite[Definition 6.46]{bongiorno_toric_stacks}). Moreover for every $p$ such that $X_{\Z_p}$ is a tame DM stack, the following equality holds:
\[
\mu_{X,p}\bigl(X(\Q_p)\bigr)
= \frac{\sharp X(\F_p)}{p^{\dim(X)}}  + \frac{1}{p} \cdot
\sum\limits_{\stackS \in \twistsector} 
\frac{\sharp\, \stackS\!\left(\F_p\right)}{p^{\dim(\stackS)}}
\]where $\sharp$ means the groupoid cardinality (see \cite[Theorem 6.52]{bongiorno_toric_stacks}).
\end{cor}

\subsection{Outline of the article}

In the second section of this article, we extend the results of theorem \ref{theorem_multi_height} obtained by the author in \cite{bongiorno2024multiheightanalysisrationalpoints} and in \cite{bongiorno_toric_stacks} to the case where the multi-height is bounded above but not below. In the third section, we apply the general version of the hyperbola method developed by Pieropan and Schindler in \cite[Theorem 1.1]{pieropan_hyperbola_method} to study the number of rational points of bounded anti-canonical height such that the multi-height have values in $\Lambda \subset \cone^{\vee}$ a simplicial subcone. In the last section of this article, we explain how to deduce the final result from the decomposition of $\cone^{\vee}$ into simplicial subcones.

\subsection{Acknowledgements}
I would like to thank my PhD advisor, Emmanuel Peyre, for all the remarks and suggestions he made during the writing process of this article. I would also like to thank Marta Pieropan for her reading and for her interest in this work.

\section{Extension of the multi-height analysis to the case bounded above but not below}

Let $\Lambda$ be a simplicial subcone of $\cone^{\vee}$ of maximal dimension, and let $$\base = (L_1, \dots, L_\rho)$$be a basis of $\Pic(X)_{\R}$ such that $$\Lambda^{\vee} = \mathrm{cone}(L_1, \dots, L_{\rho}).$$The goal of this section is to show that the arithmetic functions $f_{\downarrow}$ and $f_{\uparrow}$ from $ \N^{\rho}$ to $ \R_+ $, defined by:
\begin{equation}\label{equation_definition_arithmetic_functions}
\begin{aligned}
f_{\downarrow}(y)
&=
\sharp \Bigl\{
P \in T(\Q)
\;\Big|\;
h(P) \in \Lambda
\ \text{and}\
\lfloor H_{L_i}(P) \rfloor = y_i
\ \forall i \in \{1,\ldots,\rho\}
\Bigr\}, \\
f_{\uparrow}(y)
&=
\sharp \Bigl\{
P \in T(\Q)
\;\Big|\;
h(P) \in \Lambda
\ \text{and}\
\lceil H_{L_i}(P) \rceil = y_i
\ \forall i \in \{1,\ldots,\rho\}
\Bigr\}.
\end{aligned}
\end{equation}
verify the first hypothesis of Theorem \ref{theorem_hyperbola_method} in order to apply the hyperbola method to estimate $$\sharp T(\Q)_{h \in \Lambda \mid h(P)(\w_{X}^{-1}) \leqslant \log(B)} .$$

\subsection{Estimation of the  multi-height behaviour with values in a simplicial subcone}

For $B = (B_1, \dots, B_{\rho}) \in (\R_{\geqslant 1})^{\rho}$, we denote $u_B \in \cone^{\vee}$ the vector such that $u_B(L_i) = \log(B_i)$ for any $i$. We want to estimate the cardinality of $$T(\Q)_{h \in \Lambda \cap (-\Lambda)_B }= \left\{ P \in T(\Q) \mid \forall \ i,  1 \leqslant H_{L_i}(P) \leqslant B_i \right\} $$ where $$ (-\Lambda)_B = -\Lambda + u_B .$$

To this end, we begin by decomposing 

$$- \Lambda = \{ a \in \Pic(X)^{\vee}_{\R} \mid a(L_i) \leqslant 0 \}$$ 
into a union of compact polyhedra $\{ \mathrm{D}_n \}_{n \in (\N^*)^{\rho}}$ where 
$$\mathrm{D}_n = \{ a \in \Pic(X)^{\vee}_{\R} \mid - n_i  \beta_i \leqslant a(L_i) \leqslant - (n_i-1) \beta_i \}$$ 
with $\beta = (\beta_1, \dots, \beta_{\rho})_{\base} \in \Pic(X)^{\vee}_{\R}$ which will be suitably chosen. For $B = (B_1,..,B_{\rho}) \in \R_{\geqslant 1}^{\rho}$, we define 
$$\mathrm{D}_{n,B} = \mathrm{D}_n + u_B.$$We will deduce the estimate of $\sharp T(\Q)_{h \in \Lambda \cap (-\Lambda)_B}$ from the known estimates of $\sharp T(\Q)_{h \in \mathrm{D}_{n,B}}$. Note that since $$\sharp T(\Q)_{h \in \Lambda \cap (-\Lambda)_{\lfloor B \rfloor}} \leqslant \sharp T(\Q)_{h \in \Lambda \cap (-\Lambda)_B} \leqslant \sharp T(\Q)_{h \in \Lambda \cap (-\Lambda)_{\lceil B \rceil}},$$where $\lfloor B \rfloor = (\lfloor B_1 \rfloor,..,\lfloor B_{\rho} \rfloor ) $ and $\lceil B \rceil = (\lceil B_1 \rceil,..,\lceil B_{\rho} \rceil ) $, it is sufficient to consider the case where $B \in (\N^*)^{\rho}$, in order to get the estimate for $B \in (\R^*)^{\rho} $.

\begin{lemma}
There exists $\beta \in  \Lambda^{\circ}$, in the interior of $\Lambda$, such that for all $n \in (\N^*)^{\rho}$, for all $B \in (\N^*)^{\rho}$, and for all $P \in T(\Q)$, 
$$h(P) \neq - n \cdot \beta + u_B .$$ 
\end{lemma}

\begin{proof}
Since $T$ satisfies Northcott's property for the anticanonical height, the set $T(\Q)$ can be written as a countable union of finite sets 
$$T(\Q)_{h(P)(\w_{X}^{-1}) \leqslant N},$$ 
where $N \in \N^*$. Hence, $h(T(\Q))$ is countable, and it is possible to choose $\beta \in \Lambda^{\circ}$ such that 
$$h(T(\Q)) \cap \{  - n \cdot \beta + u_B \mid n \in (\N^*)^{\rho}, B \in (\N^*)^{\rho}\} = \varnothing.$$
\end{proof}

We choose such $\beta \in \Lambda^{\circ}$ and then we have:
$$\sharp T(\Q)_{h \in \Lambda \cap (-\Lambda)_B} = \sum\limits_{n \in (\N^*)^{\rho}} \sharp T(\Q)_{h \in \Lambda \cap \mathrm{D}_{n,B}} $$ for any $B \in (\N^*)^{\rho}$.

\subsubsection{First estimate}

Let $a_1,..,a_r \in \R$ and $b_1,..,b_r \in \R$ such that $a_i < b_i$ for any $i$. We write $\mathrm{D}(a,b)$ for the compact subset of $\Pic(X)^{\vee}_{\R}$ with $$\mathrm{D}(a,b) = \{ \varphi \in \Pic(X)^{\vee}_{\R} \mid a_i \leqslant \varphi(L_i) \leqslant b_i \ \forall i\}  $$ and we set $$\mathrm{D}(a,b)_B = \mathrm{D}(a,b) + u_B  .$$We also need to introduce the following notation:
$$
\mathcal{D}(a,b)_B := \{\, y \in \stackT(\R) \mid h_{\stackT,\infty}(y) \in \mathrm{D}(a,b)_B \,\}.
$$The same method as in \cite[Corollary 4.8 and Theorem 4.16]{bongiorno2024multiheightanalysisrationalpoints} and in \cite[Corollary 7.8 and Theorem 7.16]{bongiorno_toric_stacks}, which relies on Davenport’s result on the geometry of numbers (see \cite{davenport_geometry_numbers} and \cite[Proposition 24]{bhargava_counting_cubic_field}) yields the following result:

\begin{theorem}\label{theorem_multi_height_simplicial_cone}
    There exists $\Delta > 0$ such that $$\sharp\left( T(\Q) \right)_{h \in \mathrm{D}(a,b)_B} = \nu(\mathrm{D}(a,b) ) . \tau(X) .  \prod\limits_{i = 1}^{\rho} B_i^{\langle \w_X^{-1},L_i^* \rangle} \left( 1 + O\left(\min\limits_{1 \leqslant i \leqslant \rho}(B_i)^{- \Delta} \right)\right) $$where the implied constant depends only on $\sharp \Sigma(1)$, on the line bundles $L_i$ and on the diameter of $\mathcal{D}(a,b)_1$.
\end{theorem}

\begin{proof}
  Denote by \(e^{u_B}\in T_{\mathrm{NS}}(\R)\) the image of \(u_B\) under the exponential map
\(\operatorname{Pic}(X)^{\vee}_{\R}\to T_{\mathrm{NS}}(\R)\).
To prove the theorem it is enough to observe that
\[
e^{\langle [D_{\lambda}],u_B \rangle}
= \prod_{i = 1}^{\rho} B_i^{\langle [D_{\lambda}],L_i^* \rangle}.
\]
The dependence of the error term is a consequence of Davenport's geometry of numbers result
(see \cite[Proposition 4.1]{bongiorno2024multiheightanalysisrationalpoints}).
Hence it is enough to replace
\[
B^{\min\limits_{\lambda \in \Sigma(1)} \langle [D_{\lambda}] , u \rangle}
\]
in \cite{bongiorno2024multiheightanalysisrationalpoints} (respectively in \cite{bongiorno_toric_stacks})
by
\[
e^{\min\limits_{\lambda \in \Sigma(1)} \langle [D_{\lambda}] , u_B \rangle}.
\]
With this substitution one can adapt the lattice-point counting estimates of the proof of 
\cite[Theorem 4.6]{bongiorno2024multiheightanalysisrationalpoints}
(respectively \cite[Theorem 7.6]{bongiorno_toric_stacks}) and the Möbius inversion arguments
(see \cite[Theorem 4.14]{bongiorno2024multiheightanalysisrationalpoints} and \cite[Theorem 7.13]{bongiorno_toric_stacks})
to obtain the desired asymptotic.

Finally, set
\[
a = \min_{\lambda \in \Sigma(1)} \langle [D_{\lambda}] , \sum_{i = 1}^{\rho} L_i^* \rangle  > 0.
\]
Then we have the inequality
\[
e^{\min\limits_{\rho \in \Sigma(1)} \langle [D_{\rho}] , u_B \rangle}
\geqslant
\bigg(\min_{1 \leqslant i \leqslant \rho} B_i\bigg)^{a},
\]
which yields the claimed result, if we take $\Delta = \frac{1}{4} \cdot a$.

\end{proof}

Now, we observe that for sufficiently large $n$ (whose size can be controlled in terms of $\min(B_i)$), the set $T(\Q)_{h \in \Lambda \cap \mathrm{D}_{n,B}}$ is empty.

\begin{lemma}\label{lemma_n_assez_grand_ensemble_vide}
If there exists $i \in \{1,..,\rho\}$, such that $$n_i > \frac{1}{\beta_i}\log(B_i) + 1$$then for all $B \in (\N^*)^{\rho}$, we have: 
$$T(\Q)_{h \in \Lambda \cap \mathrm{D}_{n,B}} = \varnothing.$$
\end{lemma}

\begin{proof}
We argue by contradiction and assume that there exists $P \in T(\Q)$ such that $h(P) \in \Lambda \cap \mathrm{D}_{n,B}$. Then, for every $i \in \{1,\dots,\rho\}$, we have
\[
\log(B_i) - (n_i - 1)\beta_i \geqslant h(P)(L_i) \geqslant 0.
\]
Thus, under our assumption that there exists $i \in \{1,\dots,\rho\}$ such that
\[
n_i > \frac{1}{\beta_i}\log(B_i) + 1,
\]
we get a contradiction.
\end{proof}

Thus, we first estimate
\[
\sum_{n_1 = 1}^{\frac{1}{\beta_1}\log(B_1) + 1}
\cdots
\sum_{n_{\rho} = 1}^{\frac{1}{\beta_{\rho}}\log(B_{\rho}) + 1}
\sharp T(\Q)_{\,h \in \mathrm{D}_{n,B}}.
\]

Before proving the main result of this section, we need a bound on the remainder term of the series $\sum \nu(\mathrm{D}_n)$:

\begin{lemma}\label{lemma_majoration_reste_serie}
Let $B \in (\mathbf{N}^*)^{\rho}$. There exist constants $c>0$ and $d>0$ such that
\[
\sum_{\substack{n_1,\dots,n_\rho \ge 1\\ \exists i\in\{1,\dots,\rho\}:\,
n_i \ge \frac{1}{\beta_i}\log(B_i)+1}}
\nu(\mathrm{D}_n)
\le
\frac{c}{\min(B_i)^d}.
\]
\end{lemma}

\begin{proof}
We use the identity
\[
\nu(\mathrm{D}_n)
=
\prod_{i=1}^{\rho}
\exp\!\left(-n_i\,\beta_i\,\w_X^{-1}(L_i^*)\right)
\,\nu(\mathrm{D}_1).
\]

Since $\beta_i>0$ and $\w_X^{-1}(L_i^*)>0$ for all $i$, because
$\beta\in\Lambda^{\circ}$ and $$\w_X^{-1}\in\cone^{\circ}\subset(\Lambda^\vee)^{\circ},$$the associated series are convergent geometric series.

Estimating the remainder, we obtain
\[
\sum_{\substack{n_1,\dots,n_\rho \ge 1\\ \exists i\in\{1,\dots,\rho\}:\,
n_i \ge \frac{1}{\beta_i}\log(B_i)+1}}
\nu(\mathrm{D}_n)
\ll
\sum_{i=1}^{\rho}
\exp\!\left(-\beta_i\Bigl(\tfrac{1}{\beta_i}\log(B_i)+1\Bigr)\w_X^{-1}(L_i^*)\right).
\]

Hence,
\[
\sum_{\substack{n_1,\dots,n_\rho \ge 1\\ \exists i\in\{1,\dots,\rho\}:\,
n_i \ge \frac{1}{\beta_i}\log(B_i)+1}}
\nu(\mathrm{D}_n)
\ll
\sum_{i=1}^{\rho}
B_i^{-\w_X^{-1}(L_i^*)}.
\]

The result follows by setting
$$
d := \min_{1\le i\le \rho} \w_X^{-1}(L_i^*) > 0$$and absorbing constants into $c$.
\end{proof}

We can now state and prove the following proposition:

\begin{prop}\label{proposition_multi_height_not_bounded_below}
Let $B = (B_1,..,B_{\rho}) \in (\R_{\geqslant 1})^{\rho}$. Then there exist constants $K > 0$ and $\delta > 0$ such that
$$\left|\sum_{n_1 = 1}^{\frac{1}{\beta_1}\log(B_1) + 1}
\cdots
\sum_{n_{\rho} = 1}^{\frac{1}{\beta_{\rho}}\log(B_{\rho}) + 1}
\sharp T(\Q)_{\,h \in \mathrm{D}_{n,B}}. - \nu(-\Lambda) \, \tau(X) \, \prod\limits_{i = 1}^{\rho} B_i^{\langle \w_{X}^{-1},L_i^* \rangle} \right| \leqslant K \cdot \frac{\prod\limits_{i = 1}^{\rho} B_i^{\langle \w_{X}^{-1},L_i^* \rangle}}{\min(B_i)^{\delta}}.$$
\end{prop}

\begin{proof}
Let us write $$\mathcal{D}_{n,B} := \{ y \in \stackT(\R) \mid h_{\stackT,\infty}(y) \in \mathrm{D}_{n,B} \} .$$We have that $\mathcal{D}_{n,B} = e^{n \cdot \beta} . \mathcal{D}_{1,B}$. Hence, by theorem \ref{theorem_multi_height_simplicial_cone}, we have that there exists $K > 0$ and $\Delta > 0$ such that:
$$\left| \sharp ( T(\Q) )_{h \in \mathrm{D}_{n,B}} - \nu(\mathrm{D}_n) \, \tau(X) \, \prod\limits_{i = 1}^{\rho} B_i^{\langle \w_X^{-1},L_i^* \rangle} \right| \leqslant K \cdot \frac{\prod\limits_{i = 1}^{\rho} B_i^{\langle \w_{X}^{-1},L_i^* \rangle}}{\min(B_i)^{\Delta}}.$$
    We have:
    \begin{align*}
        &\left| \sum_{n_1 = 1}^{\frac{1}{\beta_1}\log(B_1) + 1}
\cdots
\sum_{n_{\rho} = 1}^{\frac{1}{\beta_{\rho}}\log(B_{\rho}) + 1} \sharp T(\Q)_{h \in \mathrm{D}_{n,B}} - \nu(-\Lambda) \, \tau(X) \, \prod\limits_{i = 1}^{\rho} B_i^{\langle \w_X^{-1},L_i^* \rangle} \right| \\
        &\leqslant \sum_{n_1 = 1}^{\frac{1}{\beta_1}\log(B_1) + 1}
\cdots
\sum_{n_{\rho} = 1}^{\frac{1}{\beta_{\rho}}\log(B_{\rho}) + 1} \left| \sharp ( T(\Q) )_{h \in \mathrm{D}_{n,B}} - \nu(\mathrm{D}_n) \, \tau(X) \, \prod\limits_{i = 1}^{\rho} B_i^{\langle \w_X^{-1},L_i^* \rangle} \right| \\
        &+ \, \tau(X) \, \prod\limits_{i = 1}^{\rho} B_i^{\langle \w_X^{-1},L_i^* \rangle} \cdot \sum_{\substack{n_1,\dots,n_\rho \ge 1\\ \exists i\in\{1,\dots,\rho\}:\,
n_i \ge \frac{1}{\beta_i}\log(B_i)+1}} \nu(\mathrm{D}_n)
    \end{align*}
    Hence applying theorem \ref{theorem_multi_height_simplicial_cone} and lemma \ref{lemma_majoration_reste_serie} we can deduce the theorem.
\end{proof}

\subsubsection{An upper bound}

Recall that we want to estimate the cardinality of
\[
T(\Q)_{h \in \Lambda \cap (-\Lambda)_B}
= \left\{ P \in T(\Q) \mid \forall i,\; 1 \leqslant H_{L_i}(P) \leqslant B_i \right\}.
\]
We aim to deduce this estimate from the one obtained in Proposition \ref{proposition_multi_height_not_bounded_below}. To this end, we give an upper bound, for a fixed \(k \in \{1,\ldots,\rho\}\), for the cardinality of the set
\[
\mathcal{C}_{k,n,B}
= \{ P \in T(\Q) \mid h(P) \in \mathrm{D}_{n,B} \text{ and } h(P)(L_k) \leqslant 0 \}.
\]

\begin{prop}\label{proposition_majoration_cardinal_boite_compact}
Let \(k \in \{1,\ldots,\rho\}\). Let $n \in (\N^*)^{\rho}$ with $n_i \leqslant \frac{1}{\beta_i}\log(B_i)+1$ for all $i$. Then there exist $\Delta > 0$ and a constant \(C > 0\), depending only on a fixed compact set (which depends only on $\beta$), such that
\[
\sharp \mathcal{C}_{k,n,B}
\leqslant
C \cdot \prod_{i \neq k} B_i^{\langle \w_{X}^{-1}, L_i^{*} \rangle}
\cdot e^{- \sum\limits_{i \neq k} n_i \cdot \beta(L_i) \cdot \w_{X}^{-1}(L_i^*) }.
\]
\end{prop}

\begin{proof}
By lifting to the universal torsor (respectively, to the extended universal torsor), and using \cite[Theorem 3.16]{bongiorno2024multiheightanalysisrationalpoints} for toric varieties and \cite[Theorem 6.23]{bongiorno_toric_stacks} for toric stacks, we obtain that \(\sharp \mathcal{C}_{k,n,B}\) is bounded above by the cardinality of the set of
\[
y \in \stackT(\Z) \cap T_{\Sigma(1)}(\R)
\]
satisfying, for all \(i \neq k\),
\[
- n \beta_i + \log(B_i)
\leqslant h_{\stackT,\infty}(y)(L_i)
\leqslant -(n-1) \beta_i + \log(B_i),
\]
and
\[
-\beta_k \leqslant h_{\stackT,\infty}(y)(L_k) \leqslant 0,
\]because $n_k \leqslant \frac{1}{\beta_k}\log(B_k)+1$. Recall that we write $T_{\Sigma(1)} = \G_m^{\Sigma(1)} \subset \stackT$. Let \(\mathrm{D}\) denote the compact subset of \(\Pic(X)^{\vee}_{\R}\) defined by the inequalities
\[
0 \leqslant a(L_i) \leqslant \beta_i \quad \text{for all } i \neq k,
\qquad
- \beta_k \leqslant a(L_k) \leqslant 0.
\]
For \(B = (B_1,\ldots,B_{\rho}) \in (\R_{\geqslant 1})^{\rho}\), we define a vector
\(v_B^k \in \Pic(X)^{\vee}_{\R}\) by
\[
v_B^k(L_i) = - n \cdot \beta(L_i) + \log(B_i)
\quad \text{for all } i \neq k,
\qquad
v_B^k(L_k) = 0 .
\]
We set \(\mathrm{D}_B = \mathrm{D} + v_B^k\).
We now seek to bound
\[
\sharp \left( \stackT(\Z) \cap T_{\Sigma(1)}(\R) \right)_{h_{\stackT,\infty} \in \mathrm{D}_B}.
\]

First, observe that \(h_{\stackT,\infty}^{-1}(\mathrm{D})\) is compact by
\cite[Lemma 4.3]{bongiorno2024multiheightanalysisrationalpoints} and
\cite[Theorem 6.29]{bongiorno_toric_stacks}. Hence there exists \(N > 0\) such that
\[
h_{\stackT,\infty}^{-1}(\mathrm{D})
\subset \prod_{\lambda \in \Sigma(1)} [-N,N].
\]
Since
\[
\left( \stackT(\Z) \cap T_{\Sigma(1)}(\R) \right)_{h_{\stackT,\infty} \in \mathrm{D}_B}
\subset
h_{\stackT,\infty}^{-1}(\mathrm{D}_B) \cap T_{\Sigma(1)}(\R)
=
e^{-v_B^k} \cdot \left( h_{\stackT,\infty}^{-1}(\mathrm{D}) \cap T_{\Sigma(1)}(\R) \right),
\]
we obtain
\[
h_{\stackT,\infty}^{-1}(\mathrm{D}_B) \cap T_{\Sigma(1)}(\R)
\subset
\prod_{\lambda \in \Sigma(1)}
\left(
\left[
- N \cdot e^{\langle [D_{\lambda}], v_B^k \rangle},
\,
N \cdot e^{\langle [D_{\lambda}], v_B^k \rangle}
\right]
\setminus \{0\}
\right).
\]
Using the relation
\[
\sum_{\lambda \in \Sigma(1)} [D_{\lambda}] = \w_{X}^{-1},
\]
we deduce the inequality
\[
\sharp \mathcal{C}_{k,n,B} \ll e^{\langle \w_{X}^{-1}, v_B^k \rangle}.
\]
Finally, observing that
\[
e^{\langle \w_{X}^{-1}, v_B^k \rangle}
=
\prod_{i \neq k } B_i^{\langle \w_{X}^{-1}, L_i^* \rangle}
\cdot
e^{- \sum\limits_{i \neq k} n_i \cdot \beta(L_i) \cdot \w_{X}^{-1}( L_i^*) },
\]
we obtain the announced bound.
\end{proof}

We can now give an upper bound for the difference $\sharp T(\Q)_{\sharp \mathrm{D}_{n,B}} - \sharp T(\Q)_{h \in \Lambda \cap \mathrm{D}_{n,B}}$.

\begin{prop}\label{proposition_majoration_boite_non_compact}
Let \(B \in \R_{\geqslant 1}^{\rho}\).
There exist constants \(C > 0\) and \(\Delta > 0\) such that
\[
\left|
\sum_{n_1 = 1}^{\frac{1}{\beta_1}\log(B_1) + 1}
\cdots
\sum_{n_{\rho} = 1}^{\frac{1}{\beta_{\rho}}\log(B_{\rho}) + 1}
\left(
\sharp T(\Q)_{h \in \mathrm{D}_{n,B}}
-
\sharp T(\Q)_{h \in \Lambda \cap \mathrm{D}_{n,B}}
\right)
\right|
\leqslant
C \cdot
\frac{\prod_{i = 1}^{\rho}
B_i^{\langle \w_{X}^{-1}, L_i^{*} \rangle}}
{\min\limits_{1 \leqslant i \leqslant \rho}(B_i)^{\Delta}}.
\]
\end{prop}

\begin{proof}
Since we have the inclusion
\[
T(\Q)_{h \in \mathrm{D}_{n,B}}
\setminus
T(\Q)_{h \in \Lambda \cap \mathrm{D}_{n,B}}
\subset
\bigcup_{1 \leqslant k \leqslant \rho} \mathcal{C}_{k,n,B},
\]
the result follows from Proposition \ref{proposition_majoration_cardinal_boite_compact}.
\end{proof}

We deduce the following result on the asymptotic behaviour of the number of rational points whose multi-height lies in a simplicial subcone of \(\cone^{\vee}\) and is bounded.

\begin{theorem}\label{theorem_multi_height_value_in_subcone}
Let \(B = (B_1,\ldots,B_{\rho}) \in (\R_{\geqslant 1})^{\rho}\).
Then there exist constants \(C > 0\) and \(\delta > 0\) such that
\[
\left|
\sharp ( T(\Q) )_{h \in \Lambda \cap (-\Lambda)_{B} }
-
\nu(-\Lambda)\, \tau(X)\,
\prod_{i = 1}^{\rho}
B_i^{\langle \w_{X}^{-1}, L_i^* \rangle}
\right|
\leqslant
C \cdot
\frac{\prod_{i = 1}^{\rho}
B_i^{\langle \w_{X}^{-1}, L_i^* \rangle}}
{\min(B_i)^{\delta}}.
\]
\end{theorem}

\begin{proof}
Since by our choice of $\beta$ and by Lemma \ref{lemma_n_assez_grand_ensemble_vide}, we have the equality for \\ $B \in (\N^*)^{\rho}$:
\[
T(\Q)_{h \in \Lambda \cap (-\Lambda)_B}
=
\sum_{n_1 = 1}^{\frac{1}{\beta_1}\log(B_1) + 1}
\cdots
\sum_{n_{\rho} = 1}^{\frac{1}{\beta_{\rho}}\log(B_{\rho}) + 1}
\sharp T(\Q)_{h \in \Lambda \cap \mathrm{D}_{n,B}}
\]
the result follows from
Propositions \ref{proposition_multi_height_not_bounded_below}
and \ref{proposition_majoration_boite_non_compact}, and the fact that it is enough to prove it for $B \in (\N^*)^{\rho}$ as we have seen before.
\end{proof}

\subsection{Multi-height estimates for the application of the hyperbola method}

Recall that we want to study the arithmetic functions $f_{\downarrow}$ and $f_{\uparrow}$ from $ \N^{\rho}$ to $ \R_+ $ defined by:
\begin{itemize}
    \item $f_{\uparrow}(y) = \sharp \{ P \in T(\Q) \mid h(P) \in \Lambda \text{ and } \lfloor H_{L_i}(P) \rfloor = y_i \ \forall i \in \{1,..,\rho\} \}$
    \item $f_{\downarrow}(y) = \sharp \{ P \in T(\Q) \mid h(P) \in \Lambda \text{ and } \lceil H_{L_i}(P) \rceil = y_i \ \forall i \in \{1,..,\rho\} \} .$ 
\end{itemize}We want to show that they satisfy the hypotheses of Theorem \ref{theorem_hyperbola_method} in order to apply the hyperbola method, that is, we want to estimate:
\begin{itemize}
    \item $\sum\limits_{1 \leqslant y_i \leqslant B_i,\ 1 \leqslant i \leqslant \rho} f_{\uparrow}(y) = \sharp \left\{ P \in T(\Q) \mid \forall i,\; 1 \leqslant \lceil H_{L_i}(P) \rceil \leqslant B_i \right\} ;$
    \item $\sum\limits_{1 \leqslant y_i \leqslant B_i,\ 1 \leqslant i \leqslant \rho} f_{\downarrow}(y) = \sharp \left\{ P \in T(\Q) \mid \forall i,\; 1 \leqslant \lfloor H_{L_i}(P) \rfloor \leqslant B_i \right\} .$
\end{itemize}

We have the following proposition:

\begin{prop}\label{proposition_first_hypothesis_arithmetic_function}
    Let $B_1, \dots, B_{\rho} \in \R_{\geqslant 1}$, with sufficiently large $\min\limits_{1 \leqslant i \leqslant {\rho}}(B_i)$ , we have the estimates:
\begin{itemize}
    \item $\sum\limits_{1 \leqslant y_i \leqslant B_i,\ 1 \leqslant i \leqslant \rho} f_{\uparrow}(y) 
= \nu(-\Lambda) \cdot \tau(X) \cdot \prod\limits_{i = 1}^{\rho} B_i^{\langle \w_X^{-1},L_i^* \rangle} 
+ O \!\left( \prod\limits_{i = 1}^{\rho} B_i^{\langle \w_X^{-1},L_i^* \rangle} \cdot
\min\limits_{1 \leqslant i \leqslant {\rho}}(B_i)^{-\Delta} \right) ;$ 
    \item $\sum\limits_{1 \leqslant y_i \leqslant B_i,\ 1 \leqslant i \leqslant \rho} f_{\downarrow}(y) 
= \nu(-\Lambda) \cdot \tau(X) \cdot \prod\limits_{i = 1}^{\rho} B_i^{\langle \w_X^{-1},L_i^* \rangle} 
+ O \!\left( \prod\limits_{i = 1}^{\rho} B_i^{\langle \w_X^{-1},L_i^* \rangle} \cdot
\min\limits_{1 \leqslant i \leqslant {\rho}}(B_i)^{-\Delta} \right) $.
\end{itemize}
\end{prop}

\begin{proof}
First, using the properties of the ceiling function, we obtain for
\(B \in (\N^{*})^{\rho}\) the equality
\[
\sum_{1 \leqslant y_i \leqslant B_i,\ 1 \leqslant i \leqslant \rho}
f_{\uparrow}(y)
=
\sharp T(\Q)_{h \in \Lambda \cap (-\Lambda)_B }.
\]
In this case, the desired estimate follows directly from
Theorem~\ref{theorem_multi_height_value_in_subcone}.
We conclude by observing that we have the bounds
\[
\sum_{1 \leqslant y_i \leqslant \lfloor B_i \rfloor,\ 1 \leqslant i \leqslant \rho}
f_{\uparrow}(y)
\leqslant
\sum_{1 \leqslant y_i \leqslant B_i,\ 1 \leqslant i \leqslant \rho}
f_{\uparrow}(y)
\leqslant
\sum_{1 \leqslant y_i \leqslant \lceil B_i \rceil,\ 1 \leqslant i \leqslant \rho}
f_{\uparrow}(y).
\]

Similarly, in order to estimate
\[
\sum_{1 \leqslant y_i \leqslant B_i,\ 1 \leqslant i \leqslant \rho}
f_{\downarrow}(y),
\]
it suffices to treat the case where \(B_1, \dots, B_{\rho} \in \N^*\).
Moreover, we have the inequalities
\[
\sharp T(\Q)_{h \in\Lambda \cap (-\Lambda)_B}
\leqslant
\sum_{1 \leqslant y_i \leqslant B_i,\ 1 \leqslant i \leqslant \rho}
f_{\downarrow}(y)
\leqslant
\sharp T(\Q)_{h \in \Lambda \cap (-\Lambda)_{B+ \underline{1} }},
\]
where \(\underline{1}\) denotes here the element of \(\R^{\rho}\) whose coordinates in the
canonical basis are all equal to \(1\).
The estimate is therefore again a consequence of
Theorem~\ref{theorem_multi_height_value_in_subcone}.
\end{proof}


\section{The hyperbola method}

Recall that we have chosen $\Lambda$ to be a simplicial subcone of $\cone^{\vee}$ of maximal dimension, and that $\base = (L_1, \dots, L_\rho)$ is a basis of $\Pic(X)_{\R}$ such that $$\Lambda^{\vee} = \mathrm{cone}(L_1, \dots, L_{\rho}).$$We have seen in the previous section (see Proposition \ref{proposition_first_hypothesis_arithmetic_function}) that the arithmetic function $f: \N^{\rho} \rightarrow \R_+ $ where $f = f_{\uparrow}$ or $f = f_{\downarrow}$ (see equation \ref{equation_definition_arithmetic_functions}) satisfies the first hypothesis of theorem \ref{theorem_hyperbola_method}. That is, there exist strictly positive real numbers $C_{f,M} \leqslant C_{f,E}$, $\Delta > 0$, and $\w_i > 0$ for $1 \leqslant i \leqslant \rho$ such that, 
for $B_1, \dots, B_{\rho} \in \R_{\geqslant 1}$, we have
$$
\sum\limits_{1 \leqslant y_i \leqslant B_i,\ 1 \leqslant i \leqslant \rho} f(y) 
= C_{f,M} \prod\limits_{i = 1}^{\rho} B_i^{\w_i} 
+ O \!\left( C_{f,E} \prod\limits_{i = 1}^{\rho} B_i^{\w_i} 
\min\limits_{1 \leqslant i \leqslant {\rho}}(B_i)^{-\Delta} \right)
$$as $ \min(B_i) \rightarrow \infty$, where the implied constant is independent of $f$. 

In our case, for any $1 \leqslant i \leqslant \rho$, we have
\[
\omega_i = \langle \omega_X^{-1}, L_i^* \rangle > 0,
\]
since $\Lambda \subset \cone^{\vee}$ and $\omega_X^{-1} \in \cone^{\circ}$.
Moreover, we have $C_{f,M} = \nu(-\Lambda)\,\tau(X)$.
We aim to study asymptotically the cardinality of the set
\[
T(\Q)_{\,h(P)(\omega_X^{-1}) \leqslant \log(B)}.
\]
To apply the hyperbola method, we first restrict our attention to the case where
$\Lambda$ is a simplicial subcone of $\cone^{\vee}$, and consider
\[
T(\Q)_{\,h\in \Lambda \mid h(P)(\omega_X^{-1}) \leqslant \log(B)}.
\]We then have the inequalities:
\begin{equation}\label{equation_encadrement_arithmetic_function}
    \sum\limits_{\prod\limits_{i = 1}^{\rho} y_i^{\alpha_{i}} \leqslant B } f_{\downarrow}(y) \leqslant \sharp T(\Q)_{h\in \Lambda \mid h(P)(\w_{X}^{-1}) \leqslant \log(B)} \leqslant \sum\limits_{\prod\limits_{i = 1}^{\rho} y_i^{\alpha_{i}} \leqslant B } f_{\uparrow}(y)
\end{equation}where $\alpha_{i} = \langle \w_{X}^{-1}, L_i^* \rangle$ for every $1 \leqslant i \leqslant \rho$. Now in order to apply the hyperbola method from \cite{pieropan_hyperbola_method}, we need to verify two assumptions about a polyhedron $\mathcal{P}$ (see the statement in theorem \ref{theorem_hyperbola_method}). With our choice of $f$ and $\alpha_i$, the polyhedron $\mathcal{P}$ is simply defined by the conditions:
\begin{enumerate}
    \item $\sum\limits_{i = 1}^{\rho } t_i \leqslant 1$;
    \item $t_i \geqslant 0$ for any $1 \leqslant i \leqslant \rho$.
\end{enumerate}

\begin{prop}\label{proposition_assumptions_polyhedron}
    The polyhedron $\stackP$ is bounded and is not contained in a hyperplane. The face $F$  of $\mathcal{P}$ on which the function $t \mapsto \sum\limits_{i = 1}^{\rho } t_i$ attains its maximal value is not contained in a coordinate hyperplane. Moreover, we have that the maximal value is $1$ and $\dim(F) = \rho - 1$.
\end{prop}

\begin{proof}
The fact that $\stackP$ is bounded follows directly from its definition.
Moreover, the polyhedron has non-empty interior, and therefore it is not
contained in any hyperplane.
The face $F$ is naturally of dimension $\rho-1$, and the vectors of the
canonical basis are the vertices of the simplex defined by $F$.
In particular, $F$ is not contained in any coordinate hyperplane.
\end{proof}

Before applying the hyperbola method of Pieropan and Schindler,
it remains to compute the constant $c_{\stackP}$ appearing in
Theorem~\ref{theorem_hyperbola_method}. We now give its definition in our setting:

\begin{definition}\label{definition_constante_polytope_hyperbola_method}
    Let \(\stackP \subset \R^{\rho}\) be the polyhedron defined above.  For $\delta \in ]0,1[$, let $H_{\delta}$ be the affine hyperplane given by the affine equation $\sum\limits_{i = 1}^{\rho} t_i = 1 - \delta$, and let $F_{\delta}$ be the $(\rho - 1)$-dimensional polyhedron defined by 
    \[
        F_{\delta} = H_{\delta} \cap \stackP.
    \]
    We endow $H_{\delta}$ with its natural Lebesgue measure as an affine hyperplane, which we denote by $\Vol_{\rho - 1}$.  
    We then define
    \[
        c_{\stackP} = \lim_{\delta \to 0} \Vol_{\rho - 1}(F_{\delta}).
    \]
\end{definition}

\begin{remark}
    By Proposition \ref{proposition_assumptions_polyhedron}, we may apply \cite[Proposition~3.1]{pieropan_hyperbola_method}, hence the constant $c_{\stackP}$ is well defined, as in \cite[Equation~4.3]{pieropan_hyperbola_method}.
\end{remark}

We can now compute the constant $c_{\stackP}$.

\begin{prop}\label{proposition_computation_constant_polyhedron}
    We have the equality
    $$
    c_{\stackP} = \frac{1}{(\rho - 1)!} .
    $$ 
\end{prop}

\begin{proof}
It is enough to notice that for $\delta \in ]0,1[$, we have
$$
\Vol_{\rho - 1}(H_{\delta} \cap \stackP)
= (1-\delta)^{\rho-1} \, \Vol_{\rho-1}(F)
$$and then we apply the formula to compute the volume of the $(\rho-1)$-simplex $F$ in the canonical basis of $\R^{\rho}$. This completes the proof.  
\end{proof}

Let now $B \geqslant 1$.  
We set  
\[
    \mathrm{C}_B = \{ a \in \Pic(X)_{\R}^{\vee} \mid a(\w_{X}^{-1}) \leqslant \log(B) \}.
\]  
For a simplicial sub-cone $\Lambda \subset \cone^{\vee}$ of maximal dimension, we want to estimate the cardinality of  
\[
    T(\Q)_{h \in \mathrm{C}_B \cap \Lambda} = \{ P \in T(\Q) \mid h(P) \in \Lambda \text{ and } h(P)(\w_{X}^{-1}) \leqslant \log(B) \}.
\]  
We have the following theorem:

\begin{theorem}\label{theorem_hyperbola_method_simplicial_cone}
    For $B > 1$, we have the following asymptotic behaviour:
    \[
        \# T(\Q)_{h \in \mathrm{C}_B \cap \Lambda}
        \underset{B \to +\infty}{=}
        \frac{1}{(\rho - 1)!} \cdot \nu(-\Lambda) \cdot \tau(X) \cdot B \log(B)^{\rho - 1}
        + O\!\left( \log(\log(B))^{\rho} \log(B)^{\rho - 2} B \right).
    \]
\end{theorem}

\begin{proof}
    By Proposition~\ref{proposition_first_hypothesis_arithmetic_function}, the arithmetic functions $f_{\downarrow}$ and $f_{\uparrow}$ satisfy the first hypothesis of \cite[Theorem~1.1]{pieropan_hyperbola_method} and by Proposition  \ref{proposition_assumptions_polyhedron}, $\stackP$ satisfy the other assumptions of \cite[Theorem~1.1]{pieropan_hyperbola_method}, hence using the computation of Proposition \ref{proposition_computation_constant_polyhedron}, we get by applying \cite[Theorem~1.1]{pieropan_hyperbola_method} as $B \rightarrow + \infty$:
    $$\sum\limits_{\prod\limits_{i = 1}^{\rho} y_i^{\alpha_{i}} \leqslant B } f_{\downarrow}(y)  =
        \frac{1}{(\rho - 1)!} \cdot \nu(-\Lambda) \cdot \tau(X) \cdot B \log(B)^{\rho - 1}
        + O\!\left( \log(\log(B))^{\rho} \log(B)^{\rho - 2} B \right) $$and $$ \sum\limits_{\prod\limits_{i = 1}^{\rho} y_i^{\alpha_{i}} \leqslant B } f_{\uparrow}(y) =
        \frac{1}{(\rho - 1)!} \cdot \nu(-\Lambda) \cdot \tau(X) \cdot B \log(B)^{\rho - 1}
        + O\!\left( \log(\log(B))^{\rho} \log(B)^{\rho - 2} B \right).$$We conclude the proof using equation \ref{equation_encadrement_arithmetic_function}.
\end{proof}

\section{Conclusion}\label{section_conclusion}

For $B \geqslant 1$, let us write again $\mathrm{C}_B = \{ a \in \Pic(X)_{\R}^{\vee} \mid a(\w_{X}^{-1}) \leqslant \log(B) \}$. To deduce from the previous section the result on the asymptotic behaviour of
$\sharp T(\Q)_{h \in \mathrm{C}_B}$, we shall assume the following hypothesis:

\begin{hypothesis}\label{hypothesis_lower_bound_multi_height}
    For every $P \in T(\Q)$, we have $h(P) \in \cone^{\vee}$ .
\end{hypothesis}

\begin{remark}
    This hypothesis is verified by split, smooth and proper toric varieties (see Proposition \ref{proposition_lower_bound_toric_variety}) and by the weighted projective stacks (see Proposition \ref{proposition_lower_bound_weighted_stack}). The author expect this to be true for any split toric stacks.
\end{remark}

Using this hypothesis, it is enough to decompose the dual of the effective cone 
$\cone^{\vee}$ into a finite union of simplicial subcones $\{\Lambda_{j}\}_{j \in J}$ such that 
$\Lambda_j \cap \Lambda_k$ is either empty or contained in a hyperplane. Using the inclusion–exclusion 
principle, we obtain the formula
$$
\sharp T(\Q)_{h \in \mathrm{C}_B} 
= \sum_{k = 1}^n (-1)^{k-1} 
\sum_{1 \leqslant j_1 < \cdots < j_k \leqslant n}
\sharp T(\Q)_{h \in \mathrm{C}_B \bigcap_{i = 1}^k \Lambda_{j_i}}.
$$We will therefore start by highlighting a property of the heights associated with effective line bundles, 
focusing on the toric case. Then, we will conclude by an upper bound of the cardinality of 
$T(\Q)_{h \in \mathrm{C}_B \cap \Lambda_j \cap \Lambda_{j'}}$ for $j \neq j'$.

\subsection{A key property of heights associated with effective line bundles}

To extend the result to the case where the multi-height is bounded above but not below, we must restrict to an open subset of $X$ on which the logarithmic height associated with a given effective line bundle is bounded from below. Peyre has shown the following general property (see \cite[Lemma 4.5]{Peyre_beyond_height}):

\begin{lemma}
    Let $V$ be a quasi-Fano variety (see \cite[definition 2.10]{bongiorno2024multiheightanalysisrationalpoints}), e.g. a smooth and proper toric variety. Then there exists a dense open subset $U \subset V$ and $c \in \Pic(X)^{\vee}_{\R}$ such that for any $P \in U(\Q)$, $$h(P) \in c + \cone^{\vee} .$$
\end{lemma}

\begin{question}
    Can one prove a similar result for nice DM stacks ?
\end{question}

For the case of smooth, proper and split toric varieties, we shall now explain how the previous lemma is verified if we take $U = T$ the dense open torus of $X$. 

\begin{prop}\label{proposition_lower_bound_toric_variety}
    We endow the smooth, proper and split toric variety $X$ with the natural system of heights (see \cite[Definition 9.2]{salberger_torsor}), if $h : X(\Q) \rightarrow \Pic(X)^{\vee}_{\R}$ is the associated multi-height map, we have for any $P \in T(\Q)$, $$h(P) \in \cone^{\vee} .$$
\end{prop}

\begin{proof}
    This is a consequence of \cite[Proposition 10.24]{salberger_torsor}.
\end{proof}

We now extend this result to the weighted projective stack $\stackP(\w)$ for a weight $\w \in (\N^*)^{n+1}$.

\begin{prop}\label{proposition_lower_bound_weighted_stack}
Let $h : \stackP(\w)(\Q) \rightarrow \Pic(X)^{\vee}_{\R}$ be the associated multi-height map. Then for any $P \in T(\Q)$ we have
\[
h(P) \in \cone^{\vee}.
\]
\end{prop}

\begin{proof}
We use the notation of \cite[Appendix~B.1]{bongiorno_toric_stacks}, to which we refer the reader for further details. For 
\[
t \in \bigcup_{0 \le i \le n} \frac{1}{\w_i}\Z \cap [0,1[,
\]
we denote by $\stackS_t$ the unique sector such that
\[
\age_{\num}(\stackS_t,\sheafO(1)) = t .
\]

For $i \in \{0,\dots,n\}$, we denote by $[D_i]_0$ the element of the orbifold Picard group of $\stackP(\w)$ corresponding to
\[
(\sheafO(\w_i),\varphi_i),
\]
where for any
\[
t \in \bigcup_{0 \le i \le n} \frac{1}{\w_i}\Z \cap ]0,1[,
\]
we have
\[
\varphi_i(\stackS_t) = -\{-\w_i t\}.
\]

The orbifold effective cone of  $\stackP(\w)$ is then given by the cone generated by
\[
\{ [D_i]_0 \mid 0 \le i \le n \}
\cup
\{ [\stackS_t] \mid t \in \bigcup_{0 \le i \le n} \tfrac{1}{\w_i}\Z \cap ]0,1[ \}
\]
by \cite[Theorem~1.1]{darda2025orbifoldpseudoeffectiveconestoric}.

Let $P \in T(\Q)$ and let $y \in \stackT(\Z) \cap \G_m(\R)^{\Sigma(1)}$ be a lift of $P$ (keeping the notation of that article). Using \cite[Theorem~B.6]{bongiorno_toric_stacks} and the discussion following it, we have
\[
h(P)([\stackS_t]) = \log(|k_t|) \ge 0
\]
for all $t \in \bigcup_{0 \le i \le n} \frac{1}{\w_i}\Z \cap ]0,1[$.

Moreover,
\begin{align*}
[D_i]_0
&=
\w_i [\sheafO(1)]_{\stack}
-
\sum_t \bigl(\w_i t + \{-\w_i t\}\bigr)[\stackS_t] \\
&=
\w_i [\sheafO(1)]_{\stack}
-
\sum_t \lceil \w_i t\rceil[\stackS_t] \, .
\end{align*}
Therefore we have
\begin{align*}
h(P)([D_i]_0) &=  \w_i \log \left(
\max\limits_{0 \leqslant j \leqslant n}
\left|
\left(
\prod\limits_{t \in \bigcup\limits_{0 \leqslant i \leqslant n} \frac{1}{w_i}\Z \cap ]0,1[}
k_t^{m_t(j)}
\right)
y'_j
\right|^{\frac{1}{w_j}}
\right)  -
\sum_t \lceil \w_i t\rceil \log(|k_t|) \\
&\geqslant \log|y'_i| 
\end{align*}where the last inequality is a consequence of the equality $$m_t(i) = \lceil \w_i t\rceil $$for all $t \in \bigcup_{0 \le i \le n} \frac{1}{\w_i}\Z \cap ]0,1[$.

\end{proof}

\subsection{The decomposition in simplicial subcones and a crucial upper bound}

In this paragraph, we explain how to construct a decomposition of the dual cone of the effective cone, $\cone^{\vee}$, into simplicial subcones. Recall that $\w_{X}^{-1} \in \cone^{\circ} $, i.e. it lies in the interior of $\cone$ and that $\cone^{\vee}$ is finitely generated. We may consider the polytope $F$ given by the intersection of $\cone^{\vee}$ with the affine hyperplane $\langle \w_{X}^{-1}, - \rangle = 1$. By construction, $F$ is a compact convex polytope whose vertices come from the extremal directions of $\cone^{\vee}$.

We choose a suitable triangulation, i.e. a decomposition of $F$ into affine simplices of dimension $\rho - 1$.  In particular, we may write
$$
F = \bigcup_{j \in J} \Delta_j
$$
where each $\Delta_j$ is a $(\rho-1)$–simplex, the intersection $\Delta_i \cap \Delta_j$ is either empty or a simplex of dimension $(\rho - 2)$, and $J$ is a finite index set. Such a decomposition exists by \cite[Chapter~2]{DeLoeraRambauSantosTriangulations}.

For each $j \in J$, we consider the cone $\Lambda_j$ generated by $\Delta_j$.  
This is a simplicial cone of maximal dimension.  
We then have
\begin{equation}\label{equation_decomposition_simplical_subcones}
    \cone^{\vee} = \bigcup_{j \in J} \Lambda_j
\end{equation}and for $i \neq j$, the intersection $\Lambda_i \cap \Lambda_j$ is either empty or a simplicial cone of dimension $\rho - 1$.

For the remainder of the paper, we fix such a decomposition of $\cone^{\vee}$ into simplicial subcones. Now, we aim to show that one may neglect rational points of bounded anti-canonical height whose multi-height lies in the intersection of two distinct simplicial subcones $\Lambda_j$ and $\Lambda_i$ with $j \neq i$.

\begin{prop}\label{proposition_majoration_cone_nombre_points_intersection}
    Let $i \neq j \in J$ such that $\Lambda_i \cap \Lambda_j$ is non empty, i.e. it is of dimension $\rho - 1$. Let $v_1,..,v_{\rho-1}$ be a generating set of $\Lambda_i \cap \Lambda_j$ and $h$ an element not contained in the hyperplane generated by $\Lambda_i \cap \Lambda_j$. For $\delta >0$, we write $K_{\delta}$ for the simplicial cone generated by $\{v_1,..,v_{\rho-1}, \sum\limits_{i = 1}^{\rho} v_i + \delta h \} $. We suppose that $\delta$ is sufficiently small such that $\sum\limits_{i = 1}^{\rho} v_i + \delta h \in \cone^{\vee}$.
    
    Then we have the following upper bound:
    \[
    \#\, T(\Q)_{h \in \mathrm{C}_B \cap \Lambda_i \cap \Lambda_{j}} \leqslant
    \left( \frac{\tau(X)}{(\rho - 1)!} + 1 \right) \cdot \nu(-K_{\delta}) \cdot B \log(B)^{\rho - 1} \quad \text{as } B \to \infty.
    \]
\end{prop}

\begin{proof}
With our assumptions on the fixed element $\delta>0$, we have $K_\delta \subset \cone^{\vee}$.
Hence, by applying Theorem~\ref{theorem_hyperbola_method_simplicial_cone}, we obtain the claimed upper bound for $B$ sufficiently large.
\end{proof}

We can now conclude that the cardinality
$\sharp T(\Q)_{h \in \mathrm{C}_B \cap \Lambda_i \cap \Lambda_j}$
is negligible.

\begin{prop}\label{proposition_majoration_nombre_points_intersection}
Let $i \neq j \in J$ be such that $\Lambda_i \cap \Lambda_j$ is nonempty. Then
\[
\sharp T(\Q)_{h \in \mathrm{C}_B \cap \Lambda_i \cap \Lambda_j}
= o\!\left( B \log(B)^{\rho-1} \right)
\quad \text{as } B \to \infty.
\]
\end{prop}

\begin{proof}
Let $\varepsilon>0$. We use the notation of
Proposition~\ref{proposition_majoration_cone_nombre_points_intersection} and assume that $\delta$ is sufficiently small so that
\[
\sum_{i=1}^{\rho} v_i + \delta h \in \cone^{\vee}
\quad\text{and}\quad
\nu(-K_\delta) \leqslant \frac{\varepsilon}{\frac{\tau(X)}{(\rho - 1)!} + 1}.
\]
Then, for $B$ sufficiently large, we obtain
\[
\sharp T(\Q)_{h \in \mathrm{C}_B \cap \Lambda_i \cap \Lambda_j}
\leqslant \varepsilon\, B \log(B)^{\rho-1},
\]
which yields the desired result.
\end{proof}

\subsection{The computation of the precise constant}

Let us recall here the definition of the constant $\alpha(X)$ (see \cite[Definition 2.1]{peyre_duke} and \cite[Notations 2.6.1]{peyre_zeta_height_function}).

\begin{definition}\label{definition_effective_leading_constant}
The constant $\alpha(X)$ is defined by
$$
\alpha(X) = \frac{1}{(\rho - 1)!} \cdot \int_{\cone^{\vee}} e^{-\langle \w_{X}^{-1} , y \rangle } \, dy,
$$
where the Haar measure $dy$ on $\Pic(X)^{\vee}_{\R}$ is normalized so that the covolume of the dual lattice of the Picard group is one.
\end{definition}

We can now state the following computation:

\begin{prop}\label{proposition_computation_effective_constant}
    We have the following equality:
    $$\alpha(X) = \sum\limits_{j \in J} \frac{1}{(\rho -1)!} \cdot \nu(-\Lambda_j)   $$where $\{ \Lambda_j \}_{j \in J}$ is the decomposition in simplicial subcones we have chosen in the previous paragraph.
\end{prop}

\begin{proof}
    This is a direct consequence of the decomposition of equation \ref{equation_decomposition_simplical_subcones} and of definition \ref{definition_effective_leading_constant}.
\end{proof}

We can now state again and prove the theorem \ref{theorem_article}:

\begin{theorem}
    We have the following asymptotic behaviour: $$\sharp T(\Q)_{h(P)(\w_X^{-1}) \leqslant \log(B)}   \underset{B \rightarrow +\infty}{\sim} \alpha(X) \tau(X) B^{\langle \w_{X}^{-1},u \rangle} \log(B)^{\rho(X) - 1} $$
\end{theorem}

\begin{proof}
    Recall that for $B \geqslant 1$, we write $$\mathrm{C}_B = \{ a \in \Pic(X)_{\R}^{\vee} \mid a(\w_{X}^{-1}) \leqslant \log(B) \} .$$
    Since by Hypothesis \ref{hypothesis_lower_bound_multi_height}, we have for any $P \in T(\Q)$, $h(P) \in \cone^{\vee}$, we can write using our chosen decomposition in simplicial subcones and the inclusion-exclusion principle the equality: $$\sharp T(\Q)_{h \in \mathrm{C}_B} 
= \sum_{k = 1}^n (-1)^{k-1} 
\sum_{1 \leqslant j_1 < \cdots < j_k \leqslant n}
\sharp T(\Q)_{h \in \mathrm{C}_B \bigcap_{i = 1}^k \Lambda_{j_i}}.$$Now we have using proposition \ref{proposition_computation_effective_constant}:
\begin{align*}
    & \sharp T(\Q)_{h(P)(\w_X^{-1}) \leqslant \log(B)}   - \alpha(X) \tau(X) B^{\langle \w_{X}^{-1},u \rangle} \log(B)^{\rho(X) - 1} \\
    &= \sum\limits_{j \in J} \left( \sharp T(\Q)_{h \in \mathrm{C}_B \cap \Lambda_j} - \frac{1}{(\rho-1)!} \cdot \nu(-\Lambda_j) \cdot \tau(X) B^{\langle \w_{X}^{-1},u \rangle} \log(B)^{\rho(X) - 1} \right) \\
    &+ \sum_{k > 1}^n (-1)^{k-1} 
\sum_{1 \leqslant j_1 < \cdots < j_k \leqslant n}
\sharp T(\Q)_{h(P) \in \mathrm{C}_B \bigcap_{i = 1}^k \Lambda_{j_i}}
\end{align*}

We conclude using Theorem \ref{theorem_hyperbola_method_simplicial_cone} and Proposition \ref{proposition_majoration_nombre_points_intersection}.
\end{proof}


\bibliographystyle{amsplain}
\bibliography{bibliography}

@Inbook{Peyre_beyond_height,
author="Peyre, Emmanuel",
title="Chapter \rm V: Beyond Heights: Slopes and Distribution of Rational Points",
bookTitle="Arakelov Geometry and Diophantine Applications",
year="2021",
publisher="Springer International Publishing",
address="Cham",
pages="215--279",
isbn="978-3-030-57559-5",
doi="10.1007/978-3-030-57559-5_6",
url="https://doi.org/10.1007/978-3-030-57559-5_6"
}

@article{peyre_duke,
author="Peyre, Emmanuel",
title = {{Hauteurs et mesures de Tamagawa sur les variétés de Fano}},
volume = {79},
journal = {Duke Mathematical Journal},
number = {1},
publisher = {Duke University Press},
pages = {101--218},
year = {1995},
doi = {10.1215/S0012-7094-95-07904-6},
URL = {https://doi.org/10.1215/S0012-7094-95-07904-6}
}

@incollection{peyre_zeta_height_function,
     author="Peyre, Emmanuel",
     title = {Terme principal de la fonction z\^eta des hauteurs et torseurs universels},
     booktitle = {Nombre et r\'epartition de points de hauteur born\'ee},
     editor = {Peyre Emmanuel},
     series = {Ast\'erisque},
     pages = {259--298},
     publisher = {Soci\'et\'e math\'ematique de France},
     number = {251},
     year = {1998},
     mrnumber = {1679842},
     zbl = {0966.14016},
     language = {fr},
     url = {http://www.numdam.org/item/AST_1998__251__259_0/}
}

@article{Batyrev1990,
author = {Batyrev, Victor V. and Manin, Yuri},
journal = {Mathematische Annalen},
language = {fre},
number = {1-3},
pages = {27-44},
title = {Sur le nombre des points rationnels de hauteur borné des variétés algébriques.},
url = {http://eudml.org/doc/164626},
volume = {286},
year = {1990},
}

@article{FrankeManinTschinkel1989,
  author  = {Franke, Jens and Manin, Yuri I. and Tschinkel, Yuri},
  title   = {Rational points of bounded height on Fano varieties},
  journal = {Inventiones Mathematicae},
  volume  = {95},
  year    = {1989},
  number  = {2},
  pages   = {421--435},
  doi     = {10.1007/BF01393904},
  mrnumber = {0978033},
  mrclass  = {11G35 (14G25)}
}

@article{batyrev_toric_varieties,
  author  = {Batyrev, Victor V. and Tschinkel, Yuri},
  title   = {Manin's conjecture for toric varieties},
  journal = {Journal of Algebraic Geometry},
  volume  = {7},
  year    = {1998},
  number  = {1},
  pages   = {15--53},
  mrnumber = {1603610},
  mrclass  = {14G05 (11G35)}
}

@inbook{salberger_torsor,
    author = {SALBERGER, Per},
    title = {Tamagawa measures on universal torsors and points of bounded height on Fano varieties},
    publisher = {Société mathématique de France},
    book = {Nombre et répartition de points de hauteur bornée},
    year = 1998,
    pages = {91-258},
    series = {Astérique},
    volume = {251},
}

@article{davenport_geometry_numbers,
    author = {DAVENPORT, Harold},
    title = "{On a Principle of Lipschitz}",
    journal = {Journal of the London Mathematical Society},
    volume = {s1-26},
    number = {3},
    pages = {179-183},
    year = {1951},
    month = {07},
    issn = {0024-6107},
    doi = {10.1112/jlms/s1-26.3.179},
    url = {https://doi.org/10.1112/jlms/s1-26.3.179},
    eprint = {https://academic.oup.com/jlms/article-pdf/s1-26/3/179/2536641/s1-26-3-179.pdf},
}

@article{Malle2002,
  author  = {Malle, Gunter},
  title   = {On the distribution of Galois groups},
  journal = {Journal of Number Theory},
  volume  = {92},
  year    = {2002},
  pages   = {315--329}
}

@article{Malle2004,
  author  = {Malle, Gunter},
  title   = {On the distribution of Galois groups II},
  journal = {Experimental Mathematics},
  volume  = {13},
  year    = {2004},
  number  = {2},
  pages   = {129--135}
}

@article{bhargava_counting_cubic_field,
    title = {On the Davenport–Heilbronn theorems and second order terms},
    author = {Manjul BHARGAVA, Arul SHANKAR, Jacob TSIMERMAN},
    journal = {Inventiones mathematicae},
    year = {2013},
    volume={193},
    pages = {439-499}
}

@article{bhargava_mass_formula,
    author = {Bhargava, Manjul},
    title = {Mass Formulae for Extensions of Local Fields, and Conjectures on the Density of Number Field Discriminants},
    journal = {International Mathematics Research Notices},
    volume = {2007},
    pages = {rnm052},
    year = {2007},
    month = {01},
    doi = {10.1093/imrn/rnm052},
    url = {https://doi.org/10.1093/imrn/rnm052},
    eprint = {https://academic.oup.com/imrn/article-pdf/doi/10.1093/imrn/rnm052/19150310/rnm052.pdf},
}

@article{darda2024batyrevmanin,
  author  = {Darda, Ratko and Yasuda, Takehiko},
  title   = {The Batyrev–Manin conjecture for {DM} stacks},
  journal = {Journal of the European Mathematical Society},
  year    = {2024},
  doi     = {10.4171/JEMS/???}, 
  note    = {Publié en ligne (online first)},
  pages= {N/A},
}

@article{ellenberg2022heights,
  author  = {Ellenberg, Jordan S. and Satriano, Matthew and Zureick-Brown, David},
  title   = {Heights on stacks and a generalized Batyrev--Manin--Malle conjecture},
  journal = {Forum of Mathematics, Sigma},
  volume  = {11},
  year    = {2023},
  pages   = {e14},
  doi     = {10.1017/fms.2023.5},
}

@misc{loughran_santens_malle_conjecture,
  author       = {Daniel Loughran and Tim Santens},
  title        = {Malle's conjecture and {B}rauer groups of stacks},
  year         = {2024},
  eprint       = {2412.04196},
  archivePrefix= {arXiv},
  primaryClass = {math.NT},
  note         = {arXiv:2412.04196},
  url          = {https://arxiv.org/abs/2412.04196}
}

@misc{darda_yasuda_toric_stacks_batyrev,
  title        = {The Manin conjecture for toric stacks},
  author  = {Darda, Ratko and Yasuda, Takehiko},
  year         = {2023},
  eprint       = {2311.02012},
  archivePrefix= {arXiv},
  primaryClass = {math.NT},
  note         = {arXiv:2311.02012},
  url          = {https://arxiv.org/abs/2311.02012}
}

@misc{darda2025orbifoldpseudoeffectiveconestoric,
  author  = {Darda, Ratko and Yasuda, Takehiko},
  title        = {Orbifold pseudo-effective cones of toric stacks},
  year         = {2025},
  eprint       = {2508.20434},
  archivePrefix= {arXiv},
  primaryClass = {math.AG},
  note         = {arXiv:2508.20434},
  url          = {https://arxiv.org/abs/2508.20434}
}

@article{pieropan_hyperbola_method,
     author = {Marta Pieropan and Damaris Schindler},
     title = {Hyperbola method on toric varieties},
     journal = {Journal de l{\textquoteright}\'Ecole polytechnique {\textemdash} Math\'ematiques},
     pages = {107--157},
     publisher = {\'Ecole polytechnique},
     volume = {11},
     year = {2024},
     doi = {10.5802/jep.251},
     mrnumber = {4683391},
     zbl = {07811890},
     language = {en},
     url = {https://jep.centre-mersenne.org/articles/10.5802/jep.251/}
}

@misc{bongiorno2024multiheightanalysisrationalpoints,
      title={Multi-height analysis of rational points of toric varieties}, 
      author={Nicolas Bongiorno},
      year={2024},
      eprint={2412.04226},
      archivePrefix={arXiv},
      primaryClass={math.NT},
      url={https://arxiv.org/abs/2412.04226}, 
}

@misc{bongiorno_toric_stacks,
      title={Multi-height analysis of rational points of toric stacks}, 
      author={Nicolas Bongiorno},
      year={2025},
      eprint={2512.04226},
      archivePrefix={arXiv},
      primaryClass={math.NT},
      url={https://arxiv.org/abs/2412.04226}, 
}

@book{DeLoeraRambauSantosTriangulations,
  author    = {De Loera, Jes{\'u}s A. and Rambau, J{\"o}rg and Santos, Francisco},
  title     = {Triangulations},
  publisher = {Springer},
  year      = {2010},
  series    = {Algorithms and Computation in Mathematics},
  volume    = {25},
  address   = {Berlin},
  isbn      = {978-3-642-12970-8}
}

\end{document}